 \documentclass{amsart}%
 \usepackage{amssymb}
 \usepackage{amsfonts}
 \usepackage{amsmath}
 \usepackage{graphicx}
 \usepackage{hyperref}%
 \usepackage{setspace}
 \providecommand{\U}[1]{\protect\rule{.1in}{.1in}}
 \newtheorem{theorem}{Theorem}[section]
 \theoremstyle{plain}

 \newtheorem{corollary}{Corollary}[section]

 \newtheorem{example}{Example}[section]
 
 \newtheorem{lemma}{Lemma}[section]

 \newtheorem{proposition}{Proposition}[section]

 \numberwithin{equation}{section}

\def\Z{{\mathbb Z}}

\begin{document}
 	\title[{\normalsize On 2-absorbing submodules and their Amalgamations}]{{\normalsize On 2-absorbing submodules and their Amalgamations}}
 	\author{Abuzer G\"UND\"UZ}
 	\address{Department of Mathematics, Sakarya University, 54050 Sakarya, Turkey.}
 	\email{abuzergunduz@sakarya.edu.tr}
 	\ 	
 	
 	\begin{abstract}

 Let $R_1$ and $R_2$ be commutative rings with $1\neq 0,\;M$ and $N$ be unitary $R_1-$module and $R_2-$module, respectively. $f:R_1\rightarrow R_2$ be a ring homomorphism and
 $\varphi: M\rightarrow N$ be an $R-$module homomorphism. This article studied $2-$absorbing submodule that is a generalization of the concept of prime submodule. Firstly, some characterizations of $2-$absorbing submodule are presented. Then, we examine the notion of $2-$absorbing submodule in amalgamation module $M \bowtie^{\varphi} JN$. We detected when $F\bowtie^\varphi JN$ is a $2-$absorbing submodule of $M \bowtie^{\varphi} JN$ by using the isomorphism $\frac{M \bowtie^{\varphi} JN}{F \bowtie^{\varphi} JN} \cong \frac{M}{F}$ and the homomorphism $p_{\gamma}: M \bowtie^{\varphi} JN \rightarrow \varphi(M)+JN$, where $J$ be an ideal of $R_2$ and $F$ be a submodule of $M.$


 
		
 	\end{abstract}
 	\maketitle

\section{Introduction
\label{sec:introduction}}

Throughout this paper, all rings are assumed to be commutative with identity and all modules are assumed to be unital. 
Let $R$ be a ring and $M$ be an $R-$ module and $F$ a submodule of $M$. The set $S\subseteq R$ is called a multiplicatively closed subset (m.c.s.) of $R$, if $1\in S$ and $ s_1s_2\in S$ for $s_1,s_2\in S$. Let $F$ be a submodule of $M,\;K$ be a nonempty subset of $M$ and $I$ be an ideal of $R$. In this case, the set $(F:_R K)=\{r\in R:\; rK\subseteq F\}$ and $(F:_M I)=\{m\in M:\; Im\subseteq F\}$ is called the residual of $F$ by $K$ and $I$, respectively. The annihilator of $M$ that is denoted by $ann_R(M)=(0:_R M)$.  For any $m\in M$, the set $<m>=Rm=\{rm:\; \forall r\in R\}$ is cyclic submodule of $M$ and if $M=<m>$, then $M$ is called cyclic module. For any finite subset $T$ of $M$, if $M=<T>$, then $M$ is a finitely generated $R-$module. The $R-$module $M$ is called faithful if and only if $ann_R (M)$ is zero. For an element $m\in M$, the set $Ann_R(m)=\{r\in R:\; rM\subseteq N\}$ is called the annihilator of $m$ and it is an ideal of $R$. The radical of $F$ is the intersection of all prime submodules of $M$ containing $F$, denoted by $rad(F)$, for a proper submodule $F$ of $M$. If every submodule $F$ of $M$ has the form $IM$ for some ideal $I$ of $R$, then $M$ is named a multiplication module\cite{ES}.\\
A proper submodule $F$ of $M$ is called a $2-$ absorbing submodule of $M$ if whenever $a,b\in R,\;m\in M$ and $abm\in F$, then $ab\in (F:_R M)$ or $am\in F$ or $bm \in F.$ According to this definition, every prime submodule of $M$ is a $2-$ absorbing submodule of $M$, but the converse is not always true (see [\cite{DS}, Ex. 2.2] ). The concept of $2-$ absorbing submodules of $M$ was introduced by A.Y.Darani and F.Soheilnia in \cite{DS}. They got many basic properties of these classes of submodules. This is a kind of a generalizations of prime and weak prime submodules. They showed the intersection of each distinct pair of prime submodules of $M$ is a $2-$ absorbing and if $M$ is a cyclic $R-$ module and $F$ is a $2-$absorbing submodule of $M$, then either $M-rad F=P$ is a prime submodule of $M$ such that $P^2\subseteq F$ or $M-rad F=P_1\cap P_2$ and $(M-rad F)^2 \subseteq F$, where $P_1,P_2$ are the only distinct minimal prime submodules on $F.$ Finally, they proved that $F$ is $2-$ absorbing submodule of $M$ if and only if $(F:_R M)$ is a $2-$ absorbing ideal of $R$ where $M$ is a cyclic $R-$module and $F$ is a submodule of $M$.\\
Then, S. Payrovi and S. Babaei studied on $2-$absorbing submodules in \cite{PB}. They showed a new definition of $2-$absorbing submodules when $(F:_{M}I)$ is $2-$ absorbing submodules of $M,$ where $F$ is $2-$absorbing  submodules of $M$ and $I$ is an ideal of $R$. They gave $(F:_R M)$  is a prime ideal of $R$ if and only if $(F:_{R}M)$ is a prime ideal of $R$, for all $m\in M/F$ and obtained some results. Also, they showed if $F$ is a $2-$ absorbing submodule of $M$ and $M/F$ is Noetherian, then a chain of $2-$ absorbing submodules of $M$ is constructed.\\
Let $R$ be a ring and $M$ be a $R-$module. $R\propto M$ is called the idealization of $R$ (also called the trivial extension of $R$ by $M$) that is an abelian group with with multiplication defined by $(r_1,m_1)\cdot (r_2,m_2)=(r_1r_2,r_1m_2+r_2m_1)$, where $r_1,r_2\in R$ and $m_1,m_2\in M$. It was introduced by Nagata [\cite{N}, p. 2] (see more detail \cite{AW}, \cite{DKM} and \cite{KM}).
Let $R$ be a ring, $J$ is an ideal of $R$ and $M$ be an $R-$module. The set $$R\bowtie J=\{(r,r+j):\;r\in R, j\in J \}$$
is called the amalgamated duplication of $R$ along $J$, which is a subring of $R\times R$ (see \cite{DF2},\;\cite{DF}).\\
Let $R_1$ and $R_2$ be two ring and $f:R_1 \rightarrow R_2$ be a ring homomorphism, $J$ be an ideal of $R_2,\; M$ be an $R_1-$module, $N$ be an $R_2$ module (that is an $R_1-$module induced naturally by $f$) and $\varphi:M\rightarrow N$ be an $R_1-$module homomorphism.\\
Then the set $$R_1\bowtie^f J=\{(r_1,f(r_1)+j): r_1\in R_1,\; j\in J \} $$ is called the amalgamation of $R_1$ and $R_2$ along $J$ with respect to $f,$ which is a subring of $R_1\times R_2$ (see \cite{DFF} and \cite{DFF2}). This construction is a generalization of the amalgamated duplication of a ring along an ideal (introduced and studied by \cite{DF2}). This construction is obtained by the frame of pullbacks \cite{DFF}. Besides, other construction such as $R+XS(X), R+XS[[X]]$, the $D+M$ can be studied as particular cases of the amalgamation (see [\cite{DFF}, Examples 2.5. and 2.6]). Besides, in \cite{DFF2}, they studied over $A\bowtie^f J$ with particular attention to the prime spectrum, to the chain properties and the Krull dimension.\\
The set $$M\bowtie J=\{ (m,m^{'})\in M\times M: m-m^{'}\in JM \}$$ is called the duplicatoin of the $R-$module $M$ along the ideal $J$, which is an $(R\bowtie J)-$module with scalar multiplication defined by $$(r,r+j)\cdot (m,m^{'})=(rm,(r+j)m^{'}) $$ for every $r\in R,\; j\in J$ and $(m,m^{'})\in M\bowtie J.$ (see \cite{BMT})\\
In \cite{BMT}, the authors got many properties and results concerning this kind of modules.\\
Let $N$ be a submodule of an $R-$module $M$ and $J$ be an ideal of $R.$ Then, $$N\bowtie J=\{(n,m)\in N\times M: n-m\in JM \}$$ and 
$$\Bar{N}=\{(m,n)\in M\times N: m-n\in JM\} $$ are submodule of $M\bowtie J$.
Also, the set $$M\bowtie^{\varphi} JN=\{(m,\varphi(m)+n): m\in M\;and\; n\in JN \}$$ is called the amalgamation of $M$ and $N$ along $J$ with respect to $\varphi,$ which is an $(R_1\bowtie^f J)-$module with the following scalar product 
$$(r_1,f(r_1)+j)\cdot (m,\varphi(m)+n)=(r_1m, \varphi(r_1m))+f(r_1)n+j\varphi(m)+jn$$
Also, $F$ is a submodule of $M$ and $N_2$ is a $f(R)+J-$submodule of $\varphi(M)+JN$, then $$F\bowtie^\varphi JN=\{(m,\varphi(m)+n) \in M\bowtie^\varphi JN: m\in F\; and \; n\in JN \}$$ and $$\Bar{N_2}^{\varphi}=\{(m,\varphi(m)+n) \in M\bowtie^\varphi JN: \varphi(m)+n \in N_2 \}$$ are submodules of $M\bowtie^\varphi JN$ \cite{KC1}.\\
In \cite{KMSS},the authors studied some basic properties of the amalgamation of modules such that when $M\bowtie^\varphi JN$ is a Noetherian or a coherent $R\bowtie^f J-$module.\\
If $M=R_1, N=R_2$ and $\varphi= f$, then the amalgamation module, along $J$ with respect to $\varphi$ coincides with the amalgamation of rings along $J$ with respect to $f$. If we take $R=R_1=R_2,\; M'=M=N,\;f=Id_{R}$ and $\varphi=Id_{M}$, then the amalgamation of $M$ and $N$ along $J$ with respect to $\varphi$ is exactly the duplication of the $R-$module $M$ along the ideal $J$ and so we can write $F\bowtie^\varphi JN=F\bowtie J$ (see more details \cite{KMSS} and \cite{KC1}).\\
If $S_1$ is a multiplicatively closed subset of $R_1$ and $S_2$ is a multiplicatively closed subset of $R_2$, the sets $$S\bowtie^f J=\{(s_1,f(s_1)+j): s_1\in S_1, j\in J \} $$ and  $$\Bar{S_2}^f=\{(r,f(r)+j): r\in R_1, f(r)+j \in S_2 \} $$ are multiplicatively closed subset of $M_1\bowtie^f JM_2$ \cite{KC1}.\\
In \cite{KC1}, the authors gave some results when the submodule of $M\bowtie^\varphi JN$ is a $S-$version (both $S\bowtie^f J$ and $\Bar{S_2}^f $) of prime and weakly prime submodule of $M\bowtie^\varphi JN$, under some conditions using by definitions. Unlike this, we use another technique that is a special ismorphism from [ \cite{KMSS},  Proposition 2.2.(2)] to detect $2-$absorbing submodules of $M\bowtie^\varphi JN$.\\
In \cite{UTK}, they studied the concept of $S-2-$ absorbing submodule of $M$ which is a generalization of S-prime submodules and 2-absorbing submodules, where $S$ is a multiplicatively closed subset (m.c.s.) of $R$. They got many new results for $S-2-$ absorbing submodule of $M$. Clearly, many results of $2-$absorbing submodule can be obtained by taking $S=\{1_R\}$. We used this in Proposition \ref{Prop.4}(1) and Proposition \ref{Prop.Triv.2} and Corollary \ref{cor2}.\\
The amalgamation of some special ideals is studied in several article [\cite{MMZ}, \cite{MMM} and \cite{G}]. Especially, in \cite{MMM} and \cite{G}, they gave the main theorem by using isomorphisms mentioned in [\cite{DFF}, Proposition 5.1 (2)] for $n-$absorbing, strongly $n-$absorbing ideals, $1-$absorbing prime and $1-$absorbing primary ideals of amalgamation ring. We aim to detect the $2-$absorbing submodule of $M \bowtie^{\varphi} JN$ by using the isomorphism presented in [\cite{KMSS}, Proposition 2.2.(2)] and the module homomorphism \ref{homorphism} by [\cite{KMSS}, Proposition 2.2 (3)].\\
In this note, we studied the $2-$absorbing submodule and their amalgamations. This study is composed of three section. In Section \ref{section3}, we got some new results about $2-$absorbing submodule such as localizations (see, Proposition \ref{Prop.4}), the idealization of  $2-$absorbing ideal of $R(+)M$ for the ideal $I(+)F$  of $R(+)M$ (see, Proposition \ref{Prop.Triv.}), the intersection of submodules (see, Proposition \ref{Prop.3} (1)) and the union of submodules (see, Proposition \ref{Prop.3} (2)) when the family of submodules are chain of $2-$absorbing submodules. In the last Section \ref{sectoin2}, we gave the amalgamation of the $2-$absorbing submodules of $M \bowtie^{\varphi} JN$ in Theorem \ref{thm13} which is proved by using the isomorphism $\frac{M \bowtie^{\varphi} JN}{F \bowtie^{\varphi} JN} \cong \frac{M}{F}$ by [ \cite{KMSS},  Proposition 2.2.(2)] and the homomorphism $p_{\gamma}: M \bowtie^{\varphi} JN \rightarrow \varphi(M)+JN$ \ref{homorphism} by [\cite{KMSS}, Proposition 2.2 (3)].
Compared to \cite{KC1} and \cite{Kol}, we proved Theorem \ref{thm13} for the amalgamation of $2-$absorbing submodule of $M \bowtie^{\varphi} JN$ by using by above isomorphism and homomorphism, similar to [\cite{MMM}, Theorem 2.1] and [\cite{G}, Theorem 1] in amalgamation ring. 
 Also,  we obtained new results for duplication module (see, Corollary \ref{duplication}), for $M$ be a cyclic module (see, Corollary \ref{Cor.0}), the localization of $M\bowtie^\varphi JN$ (see, Corollary \ref{Cor.1}), when $M$ be a finitely generated multiplication $R-$module (see, Corollary \ref{cor2}) and when $R$ be an integral domain, $M$ a finitely generated faithfull multiplication $R-$ module and also if $M$ is a valuation module (see, Corollary \ref{Cor.4}) by  using our main Theorem \ref{thm13}.

\section{Results in $2-$absorbing submodule} \label{section3}

In this section we give some new results for $2-$absorbing submodule to use in Corollary \ref{Cor.1}.

\begin{proposition}
If $F$ is a $2-$absorbing submodule of M and $K_1$ is a submodule of $M$ such that $K_1 \nsubseteq F$. Then, we have the following statements\\
1) $(F:_{R} K_1)$ is a $2-$absorbing submodule of $M.$ \label{prop1.1}\\
2) $(F:_{R} M)$ is a $2-$absorbing submodule of $M.$
\label{Prop.1}
\end{proposition}

\begin{proof}
1) Let $a,b\in R, m\in M$ and $abm \in (F:_{R} K_1)$ and so $abmk_1=ab(mk_1)=ab\Bar{m} \in F,$ and hence $ab\in (F:_R K_1)$ for all $k_1\in K_1$ and $\Bar{m}=mk_1\in K_1$. It implies that $ab\in ((F:_{R} K_1):_R M)$ since $(F:_{R} K_1)\subseteq ((F:_{R} K_1):_R M)$.
From $F$ is a $2-$absorbing submodule of $M$, then $ab\in ((F:_{R} K_1):_R M)$ or $a\Bar{m} \in F$ or $b\Bar{m} \in F$. 
It implies that $ab\in ((F:_{R} K_1):_R M)$ or $am\in (F:_{R} K_1) $ or $bm\in (F:_{R} K_1)$, as desired.\\
2) Follows by Proposition \ref{prop1.1} (1), if we take $K_1=M$.
\end{proof}

\begin{proposition}
Let $F$ be $2-$absorbing submodule of $M$ and for $r\in R \backslash(F:_R M)$. Then $(F:_{M} r)$ is a $2-$absorbing submodule of $M$ containing $F$.
\label{Prop.2}
\end{proposition}

\begin{proof}
Let $r\in R\backslash  (F:_{R} M),\; a,b \in R,\;m\in M$ such that $abm \in (F:_{M} r)$. Then $abmr=ab(mr)=ab\Bar{m} \in F$, for $\Bar{m}=mr \in M.$
Since $F$ is $2-$absorbing submodule of $M$, then $ab \in (F:_{R} M)\subseteq ((F:_{M} r):_R M)$ and so $ab \in ((F:_{M} r): M)$ or $amr=a\Bar{m} \in F$, then $am \in (F:_{M} r)$  or $bmr=b\Bar{m} \in F$, then $bm \in (F:_{M} r)$, as desired.
\end{proof}

Let $S$ be a m.c.s. of $R_1$. By the [\cite{UTK}, Example 5], the intersection to two $2-$absorbing submodules is not necessary $2-$absorbing submodule, if we take $S=\{1_{R_1}\}$. But when the family of $\{F_i\}_{i\in \Delta}$ be a chain of $2-$absorbing submodules of $M$, we obtain the following.

\begin{proposition}
Let $\{F_i\}_{i\in \Delta}$ be a chain of $2-$absorbing submodules of $M$. Then the followings hold:\\
1. $\bigcap_{i\in \Delta}F_i$ is a $2-$absorbing submodules of $M.$\\
2. If $M$ is a finitely generated $R-$module, then $\bigcup_{i\in \Delta} F_i$ is a $2-$absorbing submodules of $M.$
\label{Prop.3}
\end{proposition}

\begin{proof}
1. Suppose that $a,b \in R,\;m\in M,\; abm\in \bigcap_{i\in \Delta}F_i $ and $am\notin \bigcap_{i\in \Delta}F_i$, $bm \notin \bigcap_{i\in \Delta}F_i $. Then we have $\exists i \in \Delta,\; am\notin F_i$ and $\exists j \in \Delta,\; bm\notin F_j$. If $F_i\subseteq F_j$, we have $bm\notin F_i$, but since $F_i$ is $2-$absorbing submodules of $M$, then $ab\in (F_i:M)$. If $F_j\subseteq F_i$, we have $am\notin F_j$ and since $F_j$ is $2-$absorbing submodules of $M$, then $ab\in (F_j:M)$.  Hence $ab \in (\bigcap_{i\in \Delta}F_i:M)$, as desired.\\
2. Since $M$ is finitely generated, then $\bigcup_{i\in \Delta}F_i$ is a proper submodule of $M.$ Assume that  $a,b\in R, m\in M,\;abm \in \bigcup_{i\in \Delta}F_i$ and $am \notin \bigcup_{i\in \Delta}F_i$ and $bm \notin \bigcup_{i\in \Delta}F_i$. Then for all $i\in \Delta$, $abm \in F_i,\; am\notin F_i$ and $bm\notin F_i$ and so $ab \in (F_i:M)\subseteq (\bigcup_{i\in \Delta}F_i:M)$, as desired.
\end{proof}

\begin{proposition}
Let $S$ be a multiplicatively closed subset of $R_1$ and $S^{-1}M$ be the module of fraction of an $R_1-$module $M.$ Then we have the followings:\\
1) If $F$ is a  $2-$absorbing submodule of $M$, such that $(F:_{R} M) \cap S= \varnothing$,
then $S^{-1}F$ is a  $2-$absorbing submodule of $S^{-1}M$.\\
2) If $S^{-1}F$ is a  $2-$absorbing submodule of $S^{-1}M$, such that $Z(M/F)\cap S=\varnothing$, 
then $F$ is a  $2-$absorbing submodule of $M$.
\label{Prop.4}

\end{proposition}

\begin{proof}
(1): Follow by [\cite{UTK}, Proposition 1(iv)], if we take $S=\{1_{R_1}\}$.\\
2) First note that $(S^{-1}F:_{S^{-1}R} S^{-1}M)=S^{-1}(F:_{R}M)$
because $Z(M/F)\cap S=\varnothing$. Let $a,b \in R,\; m\in M$  and 
$abm \in F$ and so $(\frac{abm}{1}) \in S^{-1}F$. Since $S^{-1}F$ is 
a  $2-$absorbing submodule of $S^{-1}M$, either $(\frac{a}{1})(\frac{m}{1}) \in S^{-1}F$ or $(\frac{b}{1})(\frac{m}{1}) \in S^{-1}F$ or $(\frac{ab}{1}) \in (S^{-1}F:_{S^{-1}R} S^{-1}M)$. If $\frac{ab}{1}\in (S^{-1}F:_{S^{-1}R} S^{-1}M) $ then $\frac{ab}{1}\in S^{-1}(F:_{R}M)$ that implies $ab \in (F:_{R} M)$ and we are done. Otherwise, there exists $s\in S$ such that $sam \in F$ or there exists $t\in S$ such that $tbm \in F.$ This implies $am \in F$ or $bm \in F$, since  $Z(M/F)\cap S=\varnothing$. Hence, $F$ is a $2-$absorbing submodule of $M$.
\end{proof}


Let $M$ be a $R-$module. Recall by \cite{AW}, the set  $$R\propto M=\{(r,m):\; r\in R,\; m\in M \} $$ is a ring which is called the idealization of $R$ (also called trivial extension of $R$ by $M$) and defined by the following operation:\\
$$(r_1,m_1)\cdot(r_2,m_2)=(r_1r_2,r_1m_2+r_2m_1).$$ Here $(1,0)$ is the multiplicative identity and $(0,0)$ is zero of rings. Let $I$ be an ideal of $R$ and $F$ be a submodule of $M$, we have $I\propto F$ is an ideal of $R\propto M$ if and only if $IM\subset F$ by \cite{AW}.\\

Let $R$ be a ring and $I$ be an $R-$ module. Recall that a nonzero proper ideal $I$ of $R$ is called $2-$ absorbing ideal of $R$ if $a,b,c\in R$ and $abc \in I$, then $ab\in I$ or $ac\in I$ or $bc\in I$, which was introduced by Badawi in \cite{B}. It was a generalization of prime ideal. 

\begin{proposition} \label{Prop.Triv.2}
    Let $R$ be a ring and $I$ be an ideal of $R$. The following statemets are equivalent:\\
(1) $I$ is an $2-$absorbing ideal of $R$.\\
(2) $I(+)M$ is an $2-$absorbing ideal of $R(+)M$.
\end{proposition}

\begin{proof}
    Follow by [\cite{UTK}, Proposition 6], if we take $S=\{ 1_{R}\}$, where $S$ is a m.c.s. of $R$.
\end{proof}

\begin{proposition} \label{Prop.Triv.}

Let $R$ be a ring and $I$ be an ideal of $R$. Suppose that $F$ is a proper submodule of $M$ such that $IM\subset F$. If $I(+)F$ is an $2-$absorbing ideal of $R(+)M$, then $I$ is an $2-$absorbing ideal of $R$.

\end{proposition}

\begin{proof}
Let $abc \in I$ for $a,b,c\in R$. For $I(+)F$ be an $2-$absorbing ideal of $R(+)M$, we have $(a,0)(b,0)(c,0)\in I(+)F$,  where $(a,0),(b,0),(c,0)\in R(+)M$. In this case, either $(a,0)(b,0)=(ab,0)\in I(+)F$ or $(a,0)(c,0)=(ac,0)\in I(+)F$ or $(b,0)(c,0)=(bc,0)\in I(+)F$. It implies that either $ab\in I$ or $ac\in I$ or $bc\in I$, as desired.
\end{proof}

\section{2-absorbing submodules of Amalgamation modules} \label{sectoin2}

To avoid repetition, let us notation for the rest of the paper. Throughout the whole paper, $R_1$ and $R_2$ be two rings, $J$ be an ideal of $R_2$ and $f:R_1\rightarrow R_2$ be a ring homomorphism. $M$ and $N$ are $R_1-$module and $R_2-$module, respectively.  $\varphi: M\rightarrow N$ is a module homomorphism. The amalgamation of $M$ and $N$ along $J$ with respect to $\varphi$ is defined by $M \bowtie^{\varphi} JN=\{(m,\varphi(m)+n):m\in M, n\in JN\}$, where $M$ be $R-$module, $N$ be $R_2-$module (and so $R_1-$module) and  $\varphi: M\rightarrow N$ is $R_1-$module homomorphism.\\
Let $F$ be a  $2-$absorbing submodule of $M$, the following set $$F \bowtie^{\varphi} JN=\{(m,\varphi(m)+n):m\in F, n\in JN\},$$ 
 a $R_1 \bowtie^f J-$ module by the following scalar product
$$(r,f(r)+j)(m,\varphi(m)+n):=(rm,\varphi(rm)+f(r)n+j\varphi(m)+jn)$$
and a submodule of $M\bowtie^\varphi JN$. Note that $\varphi(rm)=f(r)\varphi(m)$, since $\varphi$ is an $R_1-$module homomorphism (\cite{KMSS} and \cite{KC1}).\\
Also the $\overline{N_2}^{\varphi}$ is defined as 
$$\overline{N_2}^{\varphi}:=\{(m,\varphi(m)+n)\in M \bowtie^{\varphi} JN:\;\varphi(m)+n\in N_2 \}$$ which is a a submodule of $M\bowtie^\varphi JN$.\\

We will detect when $F \bowtie^{\varphi} JN$ is a  $2-$absorbing submodule of $M\bowtie^\varphi JN$ and $\overline{N_2}^{\varphi}$ is a $2-$absorbing submodule of $M \bowtie^{\varphi} JN$. We should give some lemmas for this purpose.\\
The following lemma can obtain by [\cite{UTK}, Proposition 4], if we take $S=\{1_{R_1}\}$, where $S$ be multiplicatively closed subset of $R_1$. 

\begin{lemma}
Let $M$ be a $R_1-$ module, $N$ be a $R_2-$ module and $\varphi_1:M\rightarrow N$ be a module homomorphism. Then we have the followings:\\
1) Assume that $\varphi_1$ is surjective and $F$ is an submodule of $M$ containing $Ker(\varphi_1)$. If $F$ is $2-$absorbing submodule of $M$, then $\varphi_1(F)$ is $2-$absorbing submodule of $N$.\\
2) If $N_2$ is an  $2-$absorbing submodule of $N$, then ${\varphi_1}^{-1}(N_2)$ is an  $2-$absorbing submodule of $M$.
\label{Lemma3}
\end{lemma}

The following lemmas can obtain by [\cite{UTK}, Corollary 4], if we take $S=\{1_{R_1}\}$.

\begin{lemma} \label{lemma1}
Let $R_1$ be a ring, $M$ be an $R_1-$module, $F$ and $K$ be submodules of $M$ with $K\subseteq F.$ Then $F$ is $2-$absorbing submodule of $M$ if and only if $F/K$ is a $2-$absorbing submodule of $M/K$. 
\label{lemma11}
\end{lemma}

\begin{lemma} \label{lemma12}
Let $R_1$ be a ring, $M$ be an $R_1-$module, $F$ be submodules of $M$. Then $F$ is a $2-$absorbing submodule of $M$ if and only if $(0)$ is $2-$absorbing submodule of $M/F$ .
\end{lemma}

Now, we give the our main theorem which is the core of this paper.

\begin{theorem}
Under the above notations, the following statements hold.\\
1) $F \bowtie^{\varphi} JN$ is $2-$absorbing submodule $M \bowtie^{\varphi} JN$ if and only if $F$ is $2-$absorbing submodule $M.$\\
2) $\overline{N_2}^{\varphi}$ is an $2-$absorbing submodule of $ M \bowtie^{\varphi} JN$ if and only if $N_2$ is an $2-$absorbing submodule of $\varphi(M)+JN.$

\label{thm13}
\end{theorem}

\begin{proof}
We have $\frac{M \bowtie^{\varphi} JN}{F \bowtie^{\varphi} JN} \cong \frac{M}{F}$ by [ \cite{KMSS},  Proposition 2.2.(2)]. Consider $\mu: \frac{M \bowtie^{\varphi} JN}{F \bowtie^{\varphi} JN} \rightarrow \frac{M}{F}$ the given module isomorphism.  \\
1) By assertion  of Lemma \ref{lemma12}, $F$ is an $2-$absorbing submodule of $M$ if and only if $(0)$ is $2-$absorbing submodule of $\frac{M}{F}$ and by isomorphism $\frac{M \bowtie^{\varphi} JN}{F \bowtie^{\varphi} JN} \cong \frac{M}{F}$ by [ \cite{KMSS},  Proposition 2.2.(2)] if and only if $(0)$ is $2-$absorbing submodule of $\frac{M \bowtie^{\varphi} JN}{F \bowtie^{\varphi} JN}$ and again assertion of Lemma \ref{lemma12} $F \bowtie^{\varphi} JN$ is $2-$absorbing submodule $M \bowtie^{\varphi} JN$, as desired.\\
2) $\Rightarrow:$ Let $N_2$ be a submodule of $\varphi(M)+JN$. Let's define the natural projection map

\begin{equation}
  p_{\gamma}: M \bowtie^{\varphi} JN \rightarrow \varphi(M)+JN \label{homorphism}  
\end{equation}
for $(m, \varphi(m)+n)\longmapsto \varphi(m)+n$.\\
Clearly, $p_{\gamma}$ is surjective and $Ker(p_{\gamma})=\varphi^{-1}(JN)\times \{0\}$ (see the proof of [\cite{KMSS}, Proposition 2.2 (3)]) and we have $Ker(p_{\gamma})=\varphi^{-1}(JN)\times \{0\}\subseteq \overline{N_2}^{\varphi}$ by definition.\\
Also, we get $p_{\gamma}(\overline{N_2}^{\varphi})=N_2$. Because, since $p_\gamma$ is surjective and $N_2$ is a submodule of $\varphi(M)+JN$, for all $n\in N_2$, there exists $m'\in M,\; n'\in N$ such that $n=\varphi(m')+n'$. Then $(m',\varphi(m')+n')\in \overline{N_2}^{\varphi}$. Hence $n=\varphi(m')+n'=p_\gamma((m',\varphi(m')+n'))\in p_\gamma (\overline{N_2}^{\varphi})$. Namely, we have $N_2\subseteq p_{\gamma}(\overline{N_2}^{\varphi})$. On the other hand, take $\varphi(m)+n\in p_{\gamma}(\overline{N_2}^{\varphi})$ such that $(m,\varphi(m)+n)\in \overline{N_2}^{\varphi}$ and so $\varphi(m)+n\in N_2$. That is, we have $p_{\gamma}(\overline{N_2}^{\varphi})\subseteq N_2$. Thus, if  $\overline{N_2}^{\varphi}$ is an $2-$absorbing submodule of $ M \bowtie^{\varphi} JN$, then  $N_2$ is an $2-$absorbing submodule of $\varphi(M)+JN,$ by Lemma \ref{Lemma3} (1)\\
$\Leftarrow$: Note that $p_{\gamma}^{-1}(N_2)=\overline{N_2}^{\varphi}$. Because, take $(m,\varphi(m)+n)\in p_{\gamma}^{-1}(N_2)$, then $p_\gamma((m,\varphi(m)+n))=\varphi(m)+n\in N_2$ and so $(m,\varphi(m)+n)\in \overline{N_2}^{\varphi}$. Namely, we have $p_{\gamma}^{-1}(N_2)\subseteq \overline{N_2}^{\varphi}$. For the converse, take $(m,\varphi(m)+n)\in \overline{N_2}^{\varphi}$. It implies that $p_\gamma((m,\varphi(m)+n))=\varphi(m)+n\in N_2$ and so $(m,\varphi(m)+n)\in p_{\gamma}^{-1}(N_2)$, then we get $\overline{N_2}^{\varphi}\subseteq p_{\gamma}^{-1}(N_2)$. Hence, by Lemma \ref{Lemma3} (2), as desired.
\end{proof}

\begin{corollary} \label{duplication}
Under the above notations, we get the followings\\
1) $F \bowtie J$ is $2-$absorbing submodule $M \bowtie J$ if and only if $F$ is $2-$absorbing submodule $M.$\\
2) $\overline{N_2}$ is an $2-$absorbing submodule of $ M \bowtie J$ if and only if $N_2$ is an $2-$absorbing submodule of $M.$
\end{corollary}

\begin{proof}
    Follows by Theorem \ref{thm13} if we take $\varphi=Id_M$.
\end{proof}

\begin{example}
Let $R_1=\mathbb{Z}$, $R_2=\mathbb{Q}[[X]],\; M=\mathbb{Z},\; N=\mathbb{Q}[[X]], \; f:\mathbb{Z}\rightarrow \mathbb{Q}[[X]]$ be the natural embedding, $J=X \mathbb{Q}[[X]]$ is the ideal of $\mathbb{Q}[[X]]$ and $\varphi:\mathbb{Z}\rightarrow \mathbb{Q}[[X]] $ is a module homomorphism. Let $L=\{X P(X): P\in f(A) +J=\mathbb{Z} + X \mathbb{Q}[[X]]\}$. Obviously, $L$ is an ideal of $f(A) +J=\mathbb{Z} + X \mathbb{Q}[[X]]$.
Consider $\mathbb{Z} \bowtie^f \mathbb{Q}[[X]]-$module $\mathbb{Z} \bowtie^\varphi \mathbb{Q}[[X]]$.
 In this case $0 \bowtie^{\varphi} L$ is a submodule of $\mathbb{Z} \bowtie^{\varphi} \mathbb{Q}[[X]]$, but
$0 \bowtie^{\varphi} L$ is not  $2-$ absorbing submodule of $\mathbb{Z} \bowtie^{\varphi} \mathbb{Q}[[X]].$ Indeed, for
$0 \bowtie^{\varphi} L=\{(0,X P(X)) : P(X)\in \mathbb{Z} + X \mathbb{Q}[[X]]\}$, take $(5,5)\in \mathbb{Z} \bowtie^{\varphi} \mathbb{Q}[[X]]$ and $(0,\frac{X}{5^2})\in  0 \bowtie^{\varphi} H$, we have  $(5,5)^2\cdot (0,\frac{X}{5^2})=(0,X) \in 0 \bowtie^\varphi L $ but $(5,5)(0,\frac{X}{5^2}) \notin 0 \bowtie^\varphi L$ and $(5,5)^2 \notin (0 \bowtie^{\varphi} L:_{\mathbb{Z} \bowtie^{\varphi} \mathbb{Q}[[X]]} \mathbb{Z} \bowtie^{\varphi} \mathbb{Q}[[X]]).$ Because take $(1,X)\in \mathbb{Z} \bowtie^{\varphi} \mathbb{Q}[[X]]$, then $(5,5)^2 (1,X)\notin (0 \bowtie^{\varphi} L)$, as desired.
\end{example}

\begin{example}
Let $f:\Z \rightarrow \Z$ be a ring homomorphism. $J=k\Z$ is an ideal of $\Z.\;\;M=\Z/30\Z,\; and \;\;N=\Z/30\Z$ be $\mathbb{Z}-$module. $ \varphi:\Z/30\Z \rightarrow \Z/30\Z$ be a module homomorphism. Take $F=0$ is a submodule of $M.$ Then, $ 0 \bowtie^{\varphi} JN$ is a submodule of $ M \bowtie^{\varphi} JN$, but it is not $2-$absorbing submodule of $M \bowtie^{\varphi} JN$ . Because by [\cite{DS}, Example 2.2], $F$ is not $2-$absorbing submodule of $M$, then we get $ 0 \bowtie^{\varphi} JN$ is not $2-$absorbing submodule of $ M \bowtie^{\varphi} JN$ by Theorem \ref{thm13} (1).
\end{example}

\begin{example}
    By [\cite{UTK}, Example 3], especially consider $\mathbb{Z}-$module $M=\mathbb{Z}\times \mathbb{Z}_6$, then the zero submodule $F$ is not $2-$absorbing submodule, since $6\cdot (0,\Bar{1})=(0,\Bar{0})\in F$ but $6\notin (F:_\mathbb{Z} M)=0$ and $2\cdot (0,\Bar{1})=(0,\Bar{2})\notin F$ and  $3\cdot (0,\Bar{1})=(0,\Bar{3})\notin F$. Because of this, we have $0\bowtie^\varphi JN$ is not $2-$absorbing submodule of $M \bowtie^{\varphi} JN$ by Theorem \ref{thm13} (1), where $N$ be a $\mathbb{Z}-$module and $J$ be an ideal of $\mathbb{Z}$ and $\varphi: M\rightarrow N$ is a module homomorphism.
\end{example}


\begin{corollary} \label{Cor.0}
 Under the notations of the begining of the Section \ref{sectoin2}, we have the followings.\\
1) If $F$ is the intersection of each pair of distinct prime submodules of $M$, then $F \bowtie^{\varphi} JN$ is $2-$absorbing submodule of $M \bowtie^{\varphi} JN$.\\
2) Let $M$ be a cyclic $R-$module and $F$ be a submodule of $M$. If $(F:_{R} M)$ is 2-absorbing ideal of $R,$ then $F \bowtie^{\varphi} JN$ is $2-$absorbing submodule of $M \bowtie^{\varphi} JN$.\\
3) Let $P$ is a prime ideal of $R$ and $F$ be a $P-$primary submodule of a cyclic $R-$module such that $(pM)^2 \subseteq F.$ Then $F \bowtie^{\varphi} JN$ is $2-$absorbing submodule of $M \bowtie^{\varphi} JN$.\\
\end{corollary}

\begin{proof}
    The proof $1-3$ of Corollary \ref{Cor.0} can be proved using [\cite{DS}, Theorem 2.3.(1)], [\cite{DS}, Proposition 2.9.(1)], [\cite{DS}, Theorem 2.11 )], respectively with our main Theorem \ref{thm13}. 
\end{proof}

\begin{corollary} \label{Cor1}
Under the notations of the begining of the Section \ref{sectoin2}, we have the followings.\\
1) Let $F$ be 2-absorbing submodule of $M$ and $K_1$ is a submodule of $M$ such that $K_1 \nsubseteq F$. Then $(F:_{R} K_1) \bowtie^{\varphi} JN$ is $2-$absorbing submodule of $M \bowtie^{\varphi} JN$.\\
2) Let $F$ be 2-absorbing submodule of $M$ such that $M\nsubseteq F$, then $(F:_{R} M) \bowtie^{\varphi} JN$ is $2-$absorbing submodule of $M \bowtie^{\varphi} JN$.\\
3) Let $F$ be 2-absorbing submodule of $M$ and $r\in R\backslash  (F:_{R} M)$, then $(F:_{M} r) \bowtie^{\varphi} JN$ is $2-$absorbing submodule of $M \bowtie^{\varphi} JN$.\\
4)   Let $\{F_i\}_{i\in \Delta}$ be a chain of $2-$absorbing submodules of $M$. If $F= \bigcap_{i\in \Delta}F_i$ is a $2-$absorbing submodules of $M,$ then $F \bowtie^{\varphi} JN$ is $2-$absorbing of submodule $M \bowtie^{\varphi} JN$.\\
5) Let $\{F_i\}_{i\in \Delta}$ be a chain of $2-$absorbing submodules of $M$ and $M$ is a finitely generated $R-$module. If $F=\bigcup_{i\in \Delta}F_i$ is a $2-$absorbing submodules of $M,$ then $F \bowtie^{\varphi} JN$ is $2-$absorbing submodule of $M \bowtie^{\varphi} JN$.\\
6) Let $S$ be a multiplicatively closed subset of $R$ and $S^{-1}M$ be the module of fraction of an $R-$module $M$ and $\Bar{\varphi}: S^{-1}M\rightarrow S^{-1}N$ is $R-$module homomorphism. If $F$ is a  $2-$absorbing submodule of $M$ such that $(F:_{R} M) \cap S= \varnothing$, then $S^{-1}F \bowtie^{\varphi} JN$ is $2-$absorbing submodule of $ S^{-1}M \bowtie^{\varphi} JN$.\\
7) Let $S$ be a multiplicatively closed subset of $R$ and $S^{-1}M$ be the module of fraction of an $R-$module $M$ and $\Bar{\varphi}: S^{-1}M\rightarrow S^{-1}N$ is $R-$module homomorphism. If $S^{-1}F$ is a  $2-$absorbing submodule of $S^{-1}M$ such that $Z(M/F)\cap S=\varnothing$, then $F \bowtie^{\varphi} JN$ is $2-$absorbing submodule of $M \bowtie^{\varphi} JN$.\\
\label{Cor.1}
 \end{corollary}
\begin{proof}
    
    Each statement can be proved using Proposition \ref{Prop.1} (1) , Proposition \ref{Prop.1} (2), Proposition \ref{Prop.2} (1), Proposition \ref{Prop.3} (1), Proposition \ref{Prop.3} (2), Proposition \ref{Prop.4} (1) and Proposition \ref{Prop.4} (2), respectively with our main Theorem \ref{thm13}. 
\end{proof}


Let $R$ be a ring and $S$ be an multiplicatively closed set of $R$. From by \cite{UTK} for $S=\{1_R\}$ and our main Theorem \ref{thm13}, we obtain the following statements.

\begin{corollary} \label{cor2}
Under the notations of the begining of the Section \ref{sectoin2}, we have the followings.\\
(1) Let $M$ be a finitely generated multiplication $R-$module and $F$ be a submodule of $M$. If $(F:_R M)$ is an $2-$absorbing ideal of $R$, then $F\bowtie^\varphi JN$ is a $2-$absorbing submodule of $M\bowtie^\varphi JN$. \\
(2) Let $M$ be a finitely generated multiplication $R-$module and $F$ be a submodule of $M$. If $F=IM$ for some $2-$absorbing ideal $I$ of $R$ such that $ann(M)\subseteq I$, then $F\bowtie^\varphi JN$ is a $2-$absorbing submodule of $M\bowtie^\varphi JN$.\\
(3) Let $M_i$ be an $R_i-$module for each $i=1,2$ and $M=M_1\times M_2,\; R=R_1\times R_2$. Suppose that $F_1$ is a submodule of $M_1$ and $F_2$ is a submodule of $M_2$ such that $F=F_1\times F_2$. If $F_1$ is a $2-$absorbing submodule of $M_1$ or $F_2$ is a $2-$absorbing submodule of $M_2$ or $F_1$ is a prime submodule of $M_1$ or $F_2$ is a prime submodule of $M_2$, then $F\bowtie^\varphi JN$ is a $2-$absorbing submodule of $M\bowtie^\varphi JN$.\\
(4) Let $n\geq 1$ and $M_i$ be an $R_i-$module. Let $M=M_1\times M_2 \times \cdots \times M_n $ and $ R=R_1\times R_2 \times \cdots \times R_n$. Suppose that $F_i$ is a submodule of $M_i$ for each $i\in \{1,\cdots , n\}$ and $F=F_1\times F_2\times \cdots \times F_n$. If $F_k$ is an $2-$absorbing submodule of $M_k$ for some $k\in \{1,\cdots , n\}$ or $F_{k_1}$ is an prime submodule of $M_{k_1}$ and $F_{k_2}$ is an prime submodule of $M_{k_2}$ for some $k_1,k_2\in \{1,\cdots , n\}$ and $k_1\neq k_2$, then $F\bowtie^\varphi JN$ is a $2-$absorbing submodule of $M\bowtie^\varphi JN$.
\end{corollary}

\begin{proof}
    Each statements can be seen using [\cite{UTK}, Theorem 2], [\cite{UTK}, Theorem 2], [\cite{UTK}, Proposition 7], [\cite{UTK}, Theorem 3], respectively if we take $S=\{1_R\}$, where $S$ is a m.c.s. of $R$.
\end{proof}

Let $R$ be an integral domain and $M$ be an $R-$module. If for all nonzero elements $m$ and $n$ of $M$, either $Rm\subseteq Rn$ or $Rn\subseteq Rm$, then $M$ is called a valuation module. Equivalently, for any submodules $N$ and $K$ of $M$, either $N\subseteq K$ or $K\subseteq N$ \cite{Ali}.

\begin{corollary}\label{Cor.4}
Let $R$ be an integral domain and $M$ a finitely generated faithfull multiplication $R-$ module. Also if $M$ is a valuation module, then under the notations of the begining of the Section \ref{sectoin2}, we have the followings.\\
    (1) If $F$ is a $p-$primary submodule of $M$ for some prime ideal $p$ of $R$ with $p^2M\subseteq N$, then $F\bowtie^\varphi JN$ is a $2-$absorbing submodule of $M\bowtie^\varphi JN$.\\
    (2) If $F=P$ or $P^2$ for some prime submodule $P(=M-rad\;F)$ of $M$, then $F\bowtie^\varphi JN$ is a $2-$absorbing submodule of $M\bowtie^\varphi JN$.
\end{corollary}

\begin{proof}
    $(1)$ and $(2)$ can be proved by using [\cite{DS2}, Theorem 8 (2)] and [\cite{DS2}, Theorem 8 (3)], respectively with our main theorem \ref{thm13}.
\end{proof}


\section{A Conflict of Interest Statement}
The corresponding author states that there is no conflict of interest.

 \end{document}